\documentclass[letterpaper,11pt]{amsart}

\usepackage{amsmath,amssymb,amsthm,mathrsfs,comment,units,enumitem}
\usepackage[colorlinks, linktocpage, allcolors=black,breaklinks]{hyperref}

\renewcommand{\leq}{\leqslant}
\renewcommand{\geq}{\geqslant}

\newtheoremstyle{mythm}
{\baselineskip}	
{\baselineskip}	
{}		
{}		
{\bf}	
{.\\ }		
{ }		
{}		

\theoremstyle{mythm}
\newtheorem{theorem}{Theorem}	
\newtheorem{lemma}[theorem]{Lemma}
\newtheorem{proposition}[theorem]{Proposition}
\newtheorem{corollary}[theorem]{Corollary}
\newtheorem{definition}[theorem]{Definition}

\newtheorem{question}{Question}

\newtheoremstyle{myclaim}
{.5\baselineskip}	
{.5\baselineskip}	
{}		
{}		
{\sc}	
{. }		
{ }		
{}		

\theoremstyle{myclaim}

\newcommand{\astfill}{\noindent\xleaders\hbox{$\ast$}\hfill\kern0pt}
\newcommand{\ds}{\displaystyle}
\renewcommand{\hfill}{\hspace*{\fill}}

\title{An Extension of the Baire Property}
\author[C. Caruvana]{Christopher Caruvana}
\author[R. R. Kallman]{Robert R. Kallman}
\date{\today}

\begin{document}

\begin{abstract}
The purpose of this paper is to define for every Polish space $X$ a class of sets, the $EBP(X)$-sets or the
extended Baire property sets, to work out many properties of the $EBP(X)$-sets and to show their usefulness in
analysis.  For example, a proper generalization of the Pettis Theorem is proved in this context that furnishes
a new automatic continuity result for Polish groups.  The name extended Baire property sets is reasonable
since $EBP(X)$ contains the Baire property sets $BP(X)$ and it is consistent with ZFC that the containment is
proper.

\end{abstract}

\maketitle

\tableofcontents

\section{Introduction}
The purpose of this paper is to define for every Polish space $X$ a class of sets, the $EBP(X)$ sets or the
extended Baire property sets, to work out many properties of the $EBP(X)$ sets and to show their usefulness in
analysis.  For example, a proper generalization of the Pettis Theorem is proved in this context that furnishes
a new automatic continuity result for Polish groups.  The name extended Baire property sets is reasonable
since $EBP(X)$ contains the Baire property sets $BP(X)$ and it is consistent with ZFC that the containment is
proper.

Recall some basic facts about the topology on the space of probability measures on a Polish space $X$.  Let
$\mathcal{M}(X)$ be the collection of all Borel probability measures on $X$ and let $C_{b}(X)$ be the collection of
all functions $f : X \mapsto \mathbb R$ which are continuous and bounded.  Endow $\mathcal{M}(X)$ with the coarsest
topology for which each map

\begin{displaymath}
\mu \mapsto \int f\ d\mu, \mathcal M(X) \mapsto \mathbb R,
\end{displaymath}

\noindent
where $f \in C_{b}(X)$, is continuous.  With this topology, $\mathcal{M}(X)$ is Polish by Theorem 6.5, p. 46,
\cite{partha} and the set of continuous Borel probability measures $\mathcal{M}_{c}(X)$ is a $G_{\delta}$
subset of $\mathcal{M}(X)$ by Theorem 8.1, p. 53, \cite{partha}. Moreover, with Corollary 8.1, p. 55,
\cite{partha}, $\mathcal{M}_{c}(X)$ is dense if and only if $X$ has no isolated points. From this, we see that
we can apply results from the theory of Polish spaces and Baire category to $\mathcal{M}(X)$ or
$\mathcal{M}_{c}(X)$.

The genesis for this paper came from general considerations about random measures and skepticism about their
utility. The notion of a subset of a Polish space $X$ having probability zero with respect to a small fixed
subset of $\mathcal{M}(X)$ is, in general, a rather useless one in cases in which $X$ is large
(infinite-dimensional or not locally compact). For example, every Borel probability measure is supported on a
$\sigma$-compact set. Possible exceptions of interest are measures motivated by physical considerations (e.g.,
Brownian motion) or invariance properties (Haar measure and related measures on locally compact groups or
quotient spaces). Heuristically, it is plausible to call a subset of $X$ ``{small}'' provided that it is null
with respect to ``{most}'' of $\mathcal{M}(X)$.  How can we make this notion of ``{most}'' probability
measures mathematically precise?  A probability measure on $\mathcal{M}(X)$ can then be chosen.  The objection
to this is that there is no natural choice for such a probability measure and, if $X$ is large and therefore
$\mathcal{M}(X)$ is large, such a probability measure is supported on a very thin subset of $\mathcal{M}(X)$.
It is therefore natural to turn to notions of category, a more natural and topologically invariant concept.

In this paper we introduce the concepts of residually measurable sets, residually null sets and $EBP(X)$-sets
and work out many of their associated properties.  The thesis of this paper is to prove that many theorems
that have $BP(X)$-sets in their hypotheses can be very naturally strictly generalized to have $EBP(X)$-sets in
their hypotheses. Dubins and Freedman \cite{dubins_freedman} proved two of the following results of this paper
for compact Polish $X$, though they apparently did not isolate the concepts of residually measurable sets,
residually null sets or $EBP(X)$-sets or work with these general concepts.

\section{Residually Null Sets}

\begin{definition}
Let $X$ be Polish.
Define $\mathcal N : \wp(X) \mapsto \wp(\mathcal M(X))$ by the rule
\[
	\mathcal N(A) = \{ \mu \in \mathcal M(X) : \mu^\ast(A) = 0 \}
\]
where $\mu^\ast$ is the outer measure induced by $\mu$.
We will call $\mathcal N(A)$ the annihilator of $A$.
\end{definition}

\begin{definition}
Let $X$ be a Polish space and let $A \subseteq X$. We say that $A$ is residually measurable if $A$ is
$\mu$-measurable (in the sense of Carath\'{e}odory) for a co-meager set of $\mu \in \mathcal M(X)$, $A$ is
residually null if $\mathcal N(A)$ is co-meager in $\mathcal M(X)$, or $A$ is co-residually null
if $X \setminus A$ is residually null.
\end{definition}

To witness the fact that the collection of residually null subsets of a Polish space form a $\sigma$-ideal, consider the following immediate facts about the annihilator.

\begin{lemma}
Let $X$ be Polish. If $A \subseteq B \subseteq X$, then $\mathcal N(B) \subseteq \mathcal N(A)$. Also, for a
countable family $\mathscr E$ of subsets of $X$, $\mathcal N(\bigcup \mathscr E) = \bigcap \{\mathcal N(E):E
\in \mathscr E\}$.
\label{lem:countable_additivity}
\end{lemma}

Recall the following notions.

\begin{definition}
Let $X$ be a Polish space and $A \subseteq X$. We say that $A$ is universally measurable if $A$ is
$\mu$-measurable for every $\mu \in \mathcal M(X)$. Similarly, we say that $A$ is universally null if
$\mu^\ast(A) = 0$ for every $\mu \in \mathcal M_{c}(X)$.
\end{definition}

Immediately, all universally measurable sets are residually measurable.
Moreover, if $X$ is a Polish space with no isolated points and $A$ is universally null, then $A$ is residually null.
So, in spaces of usual interest, like $\mathbb R$, we see that the $\sigma$-ideal of residually null sets is finer than the $\sigma$-ideal of universally null sets.
In fact, it is strictly finer but we will need Theorem \ref{thm:meager_is_residually_null} to witness this fact.

The following easy proposition states that the notion of residually null is invariant under homeomorphisms.

\begin{proposition}
If $\varphi: X \mapsto Y$ is a homeomorphism and $\varphi^{*}: \mathcal{M}(X) \mapsto \mathcal{M}(Y)$ is
defined to be $\varphi^{*}(\mu) = \mu \circ \varphi^{-1}$, then $\varphi^{*}$ is a homeomorphism and therefore
preserves meager sets and comeager sets.  Therefore in this sense the notion of being residually null is a
topological invariant.
\label{prop:TopologicalInvarianceOfRN}
\end{proposition}

\section{Complexity Considerations}

Here, we turn our attention to the descriptive relationships between sets and their annihilators.
To begin, recall

\begin{theorem}[Portmanteau Theorem, Theorem 6.1, p. 40, \cite{partha}]
Let $X$ be a Polish space, $\{ \mu_n : n \in \omega \} \subseteq \mathcal M(X)$, and $\mu \in \mathcal M(X)$.
Then the following are equivalent:
\begin{itemize}
\item $\mu_n \to \mu$;
\item for every closed set $F \subseteq X$, $\ds\limsup \{ \mu_n(F) : n \in \omega \} \leq \mu(F)$;
\item for every open set $U \subseteq X$, $\ds \mu(U) \leq \liminf \{ \mu_n(U) : n \in \omega \}$.
\end{itemize}
\label{theorem_portmanteau}
\end{theorem}

\begin{lemma} \label{lem:open_positive_dense}
Let $X$ be Polish and $U \subseteq X$ be open and non-empty.
Then $\mathcal N(U)$ is closed and nowhere dense in $\mathcal M(X)$.
That is, no non-empty open set is residually null.
\end{lemma}

\begin{proof}
By the Portmanteau Theorem we see that $\{ \mu \in \mathcal M(X) : \mu(U) > 0\}$ is open.
To see that it is dense, let $\mu \in \mathcal M(X)$ be arbitrary and notice that, for $x \in U$ and $\lambda \in [0,1)$, the measure $\nu_\lambda := (1-\lambda)\delta_x + \lambda \mu$ gives $U$ positive measure.
Finally, notice that $\nu_\lambda \to \mu$ as $\lambda \to 1$.
This completes the proof.
\end{proof}

The Portmanteau Theorem actually has some deeper consequences for the Borel structure of certain sets of measures.

\begin{proposition}
Let $X$ be Polish, $B \subseteq X$ be Borel, and $\varepsilon\in[0,1]$.
Then
\[
	M(B,\varepsilon) := \{ \mu \in \mathcal M(X) : \mu(B) \leq \varepsilon\}
\]
is Borel.
In particular, if $B \in \mathbf \Sigma^0_\alpha(X)$, then $M(B,\varepsilon)\in\mathbf\Pi^0_\alpha(\mathcal M(X))$ and, if $B \in \mathbf \Pi^0_\alpha(X)$, then $M(B,\varepsilon) \in \mathbf \Pi^0_{\alpha+1}(\mathcal M(X))$.
\label{prop:bounded_measures}
\end{proposition}
\begin{proof}
Let $M^\ast(B,\varepsilon) = \{ \mu \in \mathcal M(X) : \mu(B) < \varepsilon\}$.
Observe that
\begin{equation}\label{eq:measure_strictly_bounded}
	M^\ast(B,\varepsilon) = \bigcup\{M(B,\varepsilon - 2^{-k}) : k \in \omega\}.
\end{equation}

Suppose $U$ is open.
Then, by the Portmanteau Theorem, $M(U,\varepsilon)$ is closed.
Hence, by \eqref{eq:measure_strictly_bounded}, $M^\ast(U,\varepsilon)$ is an $F_\sigma$.

Now, suppose $F$ is closed.
Notice that
\[
	\mu(F) \leq \varepsilon \Longleftrightarrow 1-\varepsilon \leq \mu(X\setminus F).
\]
Hence, we see that
\[
	M(F,\varepsilon) = \mathcal M(X) \setminus M^\ast(X\setminus F,1-\varepsilon),
\]
which provides us with the fact that $M(F,\varepsilon)$ is a $G_\delta$.

Now, let $\alpha$ be any countable ordinal bigger than $1$ and suppose, inductively, that for each $\beta < \alpha$, the following hold:
\begin{itemize}
	\item
	($\forall B \in \mathbf \Sigma^0_\beta(X)$)($M(B,\varepsilon) \in \mathbf \Pi^0_\beta(\mathcal M(X))$ and $M^\ast(B,\varepsilon) \in \mathbf\Sigma^0_{\beta+1}(\mathcal M(X))$);
	\item
	($\forall B \in \mathbf \Pi^0_\beta(X)$)($M(B,\varepsilon) \in \mathbf \Pi^0_{\beta+1}(\mathcal M(X))$).
\end{itemize}

Let $B \in \mathbf \Sigma^0_\alpha(X)$ and pick $\{ B_k : k \in \omega \} \subseteq \mathbf \bigcup\{\mathbf \Pi^0_\beta(X) : \beta < \alpha\}$ so that $B = \bigcup\{B_k:k\in\omega\}$.
For each $k \in \omega$, let $\beta_k < \alpha$ be so that $B_k \in \mathbf \Pi^0_{\beta_k}(X)$ and then define $\gamma_k = \max\{ \beta_i : i \leq k\} < \alpha$.
It follows that $B_0 \cup B_1 \cup \cdots \cup B_k \in \mathbf \Pi^0_{\gamma_k}(X)$ since $\mathbf \Pi^0_{\gamma_k}(X)$ is closed under finite unions.
So
\[
	M(B_0 \cup B_1 \cup \cdots \cup B_k, \varepsilon) \in \mathbf \Pi^0_{\gamma_k+1}(\mathcal M(X)) \subseteq \mathbf \Pi^0_\alpha(\mathcal M(X))
\]
by our inductive hypothesis.

As $\mathbf \Pi^0_\alpha(\mathcal M(X))$ is closed under countable intersections,
\[
	M(B,\varepsilon) = \bigcap\left\{M\left(B_0 \cup B_1 \cup \cdots \cup B_k,\varepsilon \right): k \in\omega \right\} \in \mathbf \Pi^0_\alpha(\mathcal M(X)).
\]
From \eqref{eq:measure_strictly_bounded}, we see that $M^\ast(B,\varepsilon) \in \mathbf \Sigma^0_{\alpha+1}(\mathcal M(X))$.

Now, let $B \in \mathbf \Pi^0_{\alpha}(X)$ and notice that
\[
	M(B,\varepsilon) = \mathcal M(X) \setminus M^\ast(X\setminus B,1-\varepsilon),
\]
Since $X \setminus B \in \mathbf \Sigma^0_{\alpha}(X)$ and $M^\star(X\setminus B, 1-\varepsilon) \in \mathbf \Sigma^0_{\alpha+1}(\mathcal M(X))$ by the above calculation, we see that $M(B,\varepsilon) \in \mathbf \Pi^0_{\alpha+1}(\mathcal M(X))$.
\end{proof}

\begin{corollary}
Let $X$ be Polish and $B \subseteq X$ be Borel.
Then $\mathcal N(B)$ is Borel.
In particular, if $B \in \mathbf \Sigma^0_\alpha(X)$, $\mathcal N(B) \in \mathbf \Pi^0_\alpha(\mathcal M(X))$ and, if $B \in \mathbf \Pi^0_\alpha(X)$, then $\mathcal N(B) \in \mathbf \Pi^0_{\alpha+1}(\mathcal M(X))$.
\label{cor:Borel_has_Borel_annihilator}
\end{corollary}

The following appears as \cite[3.16]{dubins_freedman} in the context of compact metrizable spaces.

\begin{theorem}
Let $X$ be a Polish space.
If a set $A \subseteq X$ is meager, then $A$ is residually null.
\label{thm:meager_is_residually_null}
\end{theorem}
\begin{proof}
By Lemma \ref{lem:countable_additivity}, it suffices to consider the case when $A$ is closed and nowhere
dense. Let $U = X \setminus A$ and notice that $U$ is open and dense in $X$. By Theorem 6.3, p. 44,
\cite{partha}, we see that $\mathcal N(A)$ is dense and, by Corollary \ref{cor:Borel_has_Borel_annihilator},
we have that $\mathcal{N}(A)$ is a $G_{\delta}$, finishing the proof.
\end{proof}

Now we are sufficiently equipped to see that the class of residually null sets is strictly finer than the
class of universally null sets. Notice that the middle-thirds Cantor set is closed and nowhere dense but not
universally null as any continuous non-zero Borel probability on $2^\omega$ can be extended to a continuous
Borel probability on $\mathbb R$.  The same holds for any meager uncountable analytic set.

It also turns out that the class of residually measurable sets is strictly finer than the class of universally
measurable sets. Let $A \subseteq [0,1]$ be non-measurable with respect to the Lebesgue measure. Then define
$E = A \times \{0\} \subseteq [0,1]^2$. It follows that $E$ is a residually measurable subset of $[0,1]^2$
since it's meager but not universally null since we can extend the Lebesgue measure on $[0,1]$ to a continuous
Borel probability on $[0,1]^2$ in the obvious way.

The next four lemmas are general facts of interest.

\begin{lemma}
Let $X$ be a Polish space, $\mu \in \mathcal M(X)$, and $A, B \subseteq X$ so that $\mu^\ast(A \triangle B) = 0$.
If $A$ is $\mu$-measurable, $B$ is $\mu$-measurable.
\label{lem:RelativeMeasurability}
\end{lemma}

\begin{corollary}
Let $X$ be a Polish space.
If $A \subseteq X$ has the Baire property, then $A$ is residually measurable.
\end{corollary}
\begin{proof}
Let $G \subseteq X$ be a $G_\delta$ and $M \subseteq X$ be meager so that $A = G \cup M$.
Theorem \ref{thm:meager_is_residually_null} guarantees that $\mathcal N(M)$ is co-meager in $\mathcal M(X)$.
For any $\mu \in \mathcal N(M)$, we have that $\mu^\ast(A \triangle G) \leq \mu^\ast(M) = 0$.
So, by Lemma \ref{lem:RelativeMeasurability}, we see that $A$ is $\mu$-measurable for each $\mu \in \mathcal N(M)$ which establishes that $A$ is residually measurable.
\end{proof}

\begin{corollary}
Let $X$ be a Polish space.
If $A \subseteq X$ has the Baire property, then $\mathcal N(A)$ has the Baire property.
\label{cor:Baire_prop}
\end{corollary}
\begin{proof}
Let $G \subseteq X$ be a $G_\delta$ and $M \subseteq X$ be meager so that $A = G \cup M$.
Then, notice that $\mathcal N(A) = \mathcal N(G) \cap \mathcal N(M)$.
Corollary \ref{cor:Borel_has_Borel_annihilator} informs us that $\mathcal N(G)$ is Borel and Theorem \ref{thm:meager_is_residually_null} tells us that $\mathcal N(M)$ is co-meager so we conclude that $\mathcal N(A)$ is a set with the Baire property.
\end{proof}

The following partial converse to Theorem \ref{thm:meager_is_residually_null} appears as \cite[Theorem 3.17]{dubins_freedman}, again, in the context of compact metric spaces.
\begin{theorem} \label{thm:ResiduallyNullAndBaireImpliesMeager}
Let $X$ be Polish and $A \subseteq X$ be a residually null set with the Baire property.
Then $A$ is meager.
\end{theorem}

\begin{proof}
As $A$ has the Baire property, let $U$ be open so that $U \triangle A$ is meager, and suppose, by way of contradiction, that $U \neq \emptyset$.
Since $\mathcal N(A \triangle U)$ is co-meager by Theorem \ref{thm:meager_is_residually_null}, Lemma \ref{lem:open_positive_dense} provides us with the fact that
\[
	G := \mathcal N(A \triangle U) \cap \mathcal N(A) \cap \{ \mu \in \mathcal M(X) : \mu(U) >0\}
\]
is co-meager in $\mathcal M(X)$.
Let $\mu \in G$ and notice that
\[
	0 < \mu(U) \leq \mu((U \setminus A) \cup A) \leq \mu(U \setminus A) + \mu(A) \leq \mu(U \triangle A) = 0,
\]
a contradiction.
Hence, $U = \emptyset$ and we see that $A$ is meager.
\end{proof}

R. Solovay demonstrated, in \cite{Solovay}, that it is consistent with ZF+DC (ZF and Dependent Choice) that all sets of reals have the Baire property assuming the existence of an inaccessible cardinal.
Later, S. Shelah, in \cite{ShelahSolovay}, showed that the inaccessible is unnecessary to obtain a model of ZF+DC where every set of reals has the Baire property.
Hence, it is consistent with ZF+DC that all residually null sets are meager by Theorem \ref{thm:ResiduallyNullAndBaireImpliesMeager}.
Since DC is strictly weaker than the full Axiom of Choice, we are led to the natural question:

\begin{question} \label{question_Mauldin}
Is it consistent with ZFC that every residually null set is meager?
\end{question}

One could ask whether or not there is some relation to a question of \cite{mauldin}, revisited in
\cite{larson_shelah} as Question 4.1, which asks if it is consistent with ZFC that all universally measurable
sets have the Baire Property. Surely, an affirmative answer to Question 1 would settle any doubts about the
consistency with ZFC of all universally null sets having the Baire property. Assuming we're in a universe of
ZFC where all universally null sets have the Baire property, could we conclude that all residually null sets
are meager? Perhaps a bit more generally, is it the case that all residually null sets can be realized as a
union of a universally null set with a meager set?

Hitherto, we've concluded complexity limitations on the annihilator of a set $A$ assuming some complexity limitations on $A$.
Now, we'll show that complexity bounds on the annihilator of $A$ can impose other complexity bounds upon $A$ itself.

Toward such an end, combining Lemma 6.1, p. 42, \cite{partha} and Lemma 6.2, p. 42, \cite{partha}, we obtain that
\begin{lemma} \label{lem:X_homeo_Dirac}
For any Polish space $X$, the map $x \mapsto \delta_x$, $X \mapsto \mathcal M(X)$, is a homeomorphism onto its range.
Moreover, $\{ \delta_x : x \in X\}$ is a closed subspace of $\mathcal M(X)$.
\end{lemma}

\begin{proposition} \label{prop:Gamma_annihilator_implies_negGamma}
Let $X$ be a Polish space and suppose $\mathbf \Gamma$ is one of the classes $\mathbf \Sigma^i_\alpha$ or $\mathbf \Pi^i_\alpha$ where $i = 0,1$ and $\alpha \geq 1$ is a countable ordinal.
Then, if $\mathcal N(A) \in \mathbf \Gamma(\mathcal M(X))$, $A \in \neg \mathbf \Gamma (X).$
\end{proposition}

\begin{proof}
Suppose that $\mathcal N(A) \in \mathbf \Gamma(\mathcal M(X))$.
Notice that
\[
	\mathcal N(A) \cap \{ \delta_x: x \in X\} = \{\delta_x : x \in X \setminus A\}.
\]
By Lemma \ref{lem:X_homeo_Dirac}, $\{ \delta_x : x \in X\setminus A\} \in \mathbf \Gamma(X)$.
It follows that $A \in \neg \mathbf \Gamma(X)$.
\end{proof}

\begin{proposition}
Suppose $A$ is residually null and that $\mathcal N(A)$ is either analytic or co-analytic.
Then $A$ is meager.
\end{proposition}
\begin{proof}
Recall that analytic sets and co-analytic sets have the Baire property.
Then, observe that, by Proposition \ref{prop:Gamma_annihilator_implies_negGamma}, $A$ is analytic or co-analytic.
Therefore, Theorem \ref{thm:ResiduallyNullAndBaireImpliesMeager} guarantees that $A$ is meager.
\end{proof}

\begin{corollary} \label{cor:Borel_iff_annihilator_Borel}
For any Polish space $X$, a set $A \subseteq X$ is Borel if and only if $\mathcal N(A)$ is Borel.
\end{corollary}

\begin{proof}
Apply Corollary \ref{cor:Borel_has_Borel_annihilator} and Proposition \ref{prop:Gamma_annihilator_implies_negGamma}.
\end{proof}

In sum, it seems that the descriptive complexity of $A$ is intimately related to the descriptive complexity of its annihilator.
Inspired by Corollaries \ref{cor:Baire_prop} and \ref{cor:Borel_iff_annihilator_Borel}, one could ask the
following question:

\begin{question}
Is it consistent with ZFC that $A$ have the Baire property whenever its annihilator has the Baire property?
\label{question_caruvana}
\end{question}

This is a more general version of Question \ref{question_Mauldin} and we will now see that a negative answer
is consistent with ZFC.  However, it is still possible that a positive answer is consistent with ZFC.

Sets which are universally null but lack the Baire property began to bubble up into the mathematical consciousness more than a century ago.
Assuming the Continuum Hypothesis, both N. Luzin and P. Mahlo independently constructed a set, now called a Luzin set, of power continuum which has countable intersection with each meager set.
About two decades later, combined works of W. Sierpi\'{n}ski and E. Szpilrajn demonstrated that a Luzin set is actually universally null.
Since a Luzin set is non-meager and universally null, Theorem \ref{thm:ResiduallyNullAndBaireImpliesMeager} informs us that a Luzin set cannot have the Baire property.
For a more detailed discussion of this along with the relevant references, see \cite{Miller}.

As it turns out, the assumption of CH to build such sets is not absolutely necessary.
In fact, \cite{larson_shelah} provides a non-meager set which is universally null under more general conditions.
Let $\mathscr M$ be the class of all meager subsets of $\mathbb R$.
Recall that $\text{cov}(\mathscr M)$ is the least cardinal $\kappa$ so that $\mathbb R$ is covered by a union of $\kappa$-many meager sets and that $\text{cof}(\mathscr M)$ is the least cardinal $\kappa$ so that there exists a family $\mathscr F$ of cardinality $\kappa$ so that every meager set is contained in a member of $\mathscr F$.
Then, \cite[Theorem 4.3]{larson_shelah} informs us that, under the assumption $\text{cov}(\mathscr M) = \text{cof}(\mathscr M)$, there exists a universally null set which does not have the Baire property.
The condition $\text{cov}(\mathscr M) = \text{cof}(\mathscr M)$ is implied by both Martin's Axiom and, less generally, the Continuum Hypothesis.

\section{Extending the Baire Property}


\begin{definition}
Let $X$ be a Polish space.
We say that $A \subseteq X$ has the Extended Baire Property or that $A$ is an EBP-set if there exists an open set $U\subseteq X$ so that $U \triangle A$ is residually null.
We will let $EBP(X)$ be the collection of all subsets of $X$ which are EBP-sets.
\end{definition}

In the following sections we prove that many of the properties of
BP-sets are faithfully mirrored in properties of the EBP-sets.
Complete details are provided for the convenience of the reader.

\begin{proposition}
The property of being an EBP-set is a topological invariant.
\end{proposition}

\begin{proof}
Let $h: X \mapsto Y$ be a homeomorphism where $X$ and $Y$ are Polish spaces.
Suppose $A \in EBP(X)$ and pick $U$ open so that $A \triangle U$ is residually null.
Notice that $h[U]$ is open and that
\[
	h[A] \triangle h[U] = h[A \triangle U].
\]
Since the class of residually null sets is invariant under homeomorphisms by Proposition
\ref{prop:TopologicalInvarianceOfRN}, we see that $h[A] \in EBP(Y)$.
\end{proof}

\begin{lemma}
Let $A, B, C, A_\xi , B_\xi \subseteq X$ for $\xi < \kappa$ where $\kappa$ is any cardinal.
Then
\begin{itemize}
\item
$A^c \triangle B^c = A \triangle B$;
\item
$A \triangle C \subseteq (A \triangle B) \cup (B \triangle C)$;
\item
$(A \cap B) \triangle (C \cap D) \subseteq (A \triangle C) \cup (B \triangle D)$;
\item
$\left[ \bigcup \{ A_\xi : \xi < \kappa \} \right] \triangle \left[ \bigcup\{ B_\xi : \xi < \kappa \} \right] \subseteq \bigcup \{ A_\xi \triangle B_\xi : \xi < \kappa \}$.
\end{itemize}
\end{lemma}

\begin{theorem}
\label{thm:RSetSigmaAlgebra}
For a Polish space $X$, the collection $EBP(X)$ is a $\sigma$-algebra and is the smallest $\sigma$-algebra containing the open sets and the residually null sets.
\end{theorem}
\begin{proof}
First, we show that the open sets and the residually null sets are in $EBP(X)$.
If $U$ is open, then $U \triangle U = \emptyset$ so $U \in EBP(X)$.
If $A$ is residually null, then $A \triangle \emptyset = A$ so we see that $A \in EBP(X)$.

Now we will see that $EBP(X)$ is a $\sigma$-algebra.
Suppose $A \in EBP(X)$ and let $U$ be open so that $A \triangle U$ is residually null.
Then
\[
	A^c \triangle \text{int}_X(U^c)
	\subseteq (A^c \triangle U^c) \cup ( U^c \triangle \text{int}_X(U^c) )
	= (A \triangle U) \cup ( U^c \setminus \text{int}_X(U^c) ).
\]
Since $U$ is open, $U^c \setminus \text{int}_X(U^c)$ is closed and nowhere dense.
Hence, $A^c \in EBP(X)$.
Now, suppose $\{ A_n : n \in \omega\} \subseteq EBP(X)$ and let $U_n$ be open so that $A_n \triangle U_n$ is residually null for each $n \in \omega$.
Then, observe that
\[
	\left[ \bigcup\{A_n:n\in\omega\} \right] \triangle \left[ \bigcup\{U_n:n\in\omega\} \right] \subseteq \bigcup \{ A_n \triangle U_n : n \in \omega  \},
\]
providing us with the fact that $\bigcup\{A_n:n\in\omega\} \in EBP(X)$.

Now suppose $\mathscr S$ is any $\sigma$-algebra of sets of $X$ containing both the open sets and the residually null sets.
Let $A \in EBP(X)$ and let $U$ be open and $R$ be residually null so that $A \triangle U = R$.
Since $U, R \in \mathscr S$, we see that $A = U \triangle R \in \mathscr S$.
\end{proof}

An immediate observation from Theorems \ref{thm:meager_is_residually_null} and \ref{thm:RSetSigmaAlgebra} is
that $BP(X) \subseteq EBP(X)$.
Also, all residually null sets are EBP-sets which gives us that, by discussion above, the containment
$BP(X) \subseteq EBP(X)$ is proper in all models of ZFC+MA or, less generally, ZFC+CH.

\begin{lemma} \label{lem:R-set_Representations}
Let $X$ be a Polish space and $A \subseteq X$.
Then the following are equivalent:
\begin{enumerate}[label=(\roman*),ref=(\roman*)]
\item \label{RRepFirst}
$A$ is an EBP-set;
\item \label{RRepSecond}
There exists an EBP-set $B$ so that $A \triangle B$ is residually null;
\item \label{RRepThird}
$A = (U \setminus N) \cup M$ where $U$ is open and $N$ and $M$ are both residually null.
\end{enumerate}
\end{lemma}
\begin{proof}
(\ref{RRepFirst} $\Rightarrow$ \ref{RRepSecond})
By definition, if $A$ is an EBP-set, there exists an open set $U$ so that $A \triangle U$ is residually null and $U$ is an EBP-set.

(\ref{RRepSecond} $\Rightarrow$ \ref{RRepFirst})
Let $B$ be an EBP-set so that $A \triangle B =: N$ is residually null.
Then, since $A = B \triangle N$ and the family of EBP-sets is a $\sigma$-algebra by Theorem \ref{thm:RSetSigmaAlgebra}, we see that $A$ is an EBP-set.

(\ref{RRepFirst} $\Rightarrow$ \ref{RRepThird})
Let $U$ be open so that $A \triangle U =: N$ is residually null.
Then notice that, letting $M = N \setminus U$,
\[
	A = U \triangle N = (U\setminus N) \cup (N \setminus U) = (U \setminus N) \cup M,
\]
the desired form.

(\ref{RRepThird} $\Rightarrow$ \ref{RRepFirst})
This follows immediately from Theorem \ref{thm:RSetSigmaAlgebra}.
\end{proof}

\begin{theorem}
All EBP-sets are residually measurable.
\end{theorem}
\begin{proof}
Suppose $A \in EBP(X)$ and let $U$ be open so that $U \triangle A$ is residually null.
For any $\mu \in \mathcal N(U \triangle A)$, we see that $A$ is $\mu$-measurable by Lemma \ref{lem:RelativeMeasurability}.
Therefore, $A$ is residually measurable.
\end{proof}

\begin{definition}
For a Polish space $X$, define $R : \wp(X) \mapsto \wp(X)$ by
\[
	R(A) = \bigcup \{ U \text{ open} : U \cap A \text{ is residually null} \}.
\]
\end{definition}
Notice that $R(A)$ is open.
Define $CR(A) = X \setminus R(A)$ and $ICR(A) = \text{int}_X(CR(A))$.
Recall that, inspired by the set function $D$ appearing in \cite[p. 83]{kuratowski},
\[
	M(A) := \bigcup \{ U \text{ open} : U \cap A \text{ is meager}\}
\]
satisfies $M(A) \subseteq R(A)$ as all meager sets are residually null.
From $D(A) := X \setminus M(A)$, we see immediately that $CR(A) \subseteq D(A)$.

\begin{theorem} \label{thm:RNPointsOfARN}
Let $X$ be a Polish space.
For any set $A \subseteq X$, both $A \cap R(A)$ and $A \cap \text{cl}_X(R(A))$ are residually null.
\end{theorem}
\begin{proof}
Let $\mathcal B$ be a countable base for the topology on $X$ and let
\[
	\mathscr U = \{ B \in \mathcal B : B \cap A \text{ is residually null} \}.
\]
Enumerate $\mathscr U = \{ U_n : n \in \omega \}$ and define $U = \bigcup \mathscr U$.
Surely, $U \subseteq R(A)$.

Suppose $V$ is open so that $V \cap A$ is residually null.
As
\[
	V = \bigcup \{ B \in \mathcal B : B \subseteq V\}
\]
and every $B \in \mathcal B$ so that $B \subseteq V$ satisfies the property that $A \cap B \subseteq A \cap V$ is residually null, we see that $V \subseteq U$.
Hence, $R(A) \subseteq U$ which establishes that $U = R(A)$.

Now, notice that
\[
	A \cap R(A) = A \cap \bigcup \{ U_n : n \in \omega\} = \bigcup \{ A \cap U_n : n \in \omega\}.
\]
Since each $A \cap U_n$ is residually null, we see that $A \cap R(A)$ is residually null.
Moreover, as $R(A)$ is open, $\text{cl}_X( R(A) ) \setminus R(A)$ is closed and nowhere dense.
Therefore,
\[
	A \cap \text{cl}_X(R(A)) \subseteq (A \cap R(A)) \cup (\text{cl}_X( R(A) ) \setminus R(A)),
\]
finishing the proof.
\end{proof}

\begin{lemma} \label{lem:MegaLemma}
Let $X$ be a Polish space, $A,B \subseteq X$, and $A_\xi \subseteq X$ for all $\xi < \kappa$ where $\kappa$ is a cardinal.
\begin{enumerate}[label=(\roman*)]
\item \label{MegaLemma_Monotonicity}
If $A \subseteq B$, then $R(B) \subseteq R(A)$ and, consequently, $CR(A) \subseteq CR(B)$;
\item \label{MegaLemma_Additivity}
$R(A \cup B) = R(A) \cap R(B)$ which provides that
\[
	CR(A\cup B) = CR(A) \cup CR(B);
\]
\item \label{MegaLemma_RelationshipWithClosure}
$X \setminus \text{cl}_X(A) \subseteq R(A)$ which yields $CR(A) \subseteq \text{cl}_X(A)$;
\item \label{MegaLemma_OpenSetClosure}
If $U$ is open, then $CR(U) = \text{cl}_X(U)$.
Moreover, $U \subseteq ICR(U)$ and both $ICR(U) \setminus U$ and $CR(U) \setminus U$ are residually null;
\item \label{MegaLemma_RNiffCREmpty}
$A$ is residually null if and only if $CR(A) = \emptyset$;
\item \label{MegaLemma_AMinusCR}
$A \setminus CR(A)$ is residually null; 
\item \label{MegaLemma_AMinusICR}
$A \setminus ICR(A)$ is residually null; 
\item \label{MegaLemma_Differences}
$CR(A) \setminus CR(B) \subseteq CR(A\setminus B)$;
\item \label{MegaLemma_MonotonicityOverIntersections}
$CR\left( \bigcap\{A_\xi : \xi < \kappa\} \right) \subseteq \bigcap \{ CR(A_\xi) :\xi < \kappa\}$;
\item \label{MegaLemma_MonotonicityOverUnions}
$\bigcup \{ CR(A_\xi) : \xi < \kappa \} \subseteq CR \left( \bigcup\{A_\xi : \xi < \kappa\} \right)$;
\item \label{MegaLemma_IdempotenceOfCR}
$CR(CR(A)) = CR(A)$;
\item \label{MegaLemma_CRMinusA}
$A$ is an EBP-set if and only if $CR(A) \setminus A$ is residually null;
\item \label{MegaLemma_CRTriangleA}
$A$ is an EBP-set if and only if $CR(A) \triangle A$ is residually null;
\item \label{MegaLemma_ICRMinusA}
$A$ is an EBP-set if and only if $ICR(A) \setminus A$ is residually null;
\item \label{MegaLemma_ICRTriangleA}
$A$ is an EBP-set if and only if $ICR(A) \triangle A$ is residually null;
\item \label{MegaLemma_RegularityOfCR}
$CR(A) =\text{cl}_X(ICR(A))$.
\end{enumerate}
\end{lemma}

\begin{proof}
\ref{MegaLemma_Monotonicity}
Notice that $R(B)$ is open and $R(B) \cap A \subseteq R(B) \cap B$, the latter of which is residually null.
It follows that $R(B) \subseteq R(A)$ and, therefore, that $CR(A) \subseteq CR(B)$.

\ref{MegaLemma_Additivity}
From \ref{MegaLemma_Monotonicity} we see that $R(A\cup B) \subseteq R(A) \cap R(B)$.
Now, observe that
\begin{align*}
	(A\cup B) \cap (R(A)\cap R(B))
	&= (A \cap (R(A)\cap R(B))) \cup (B \cap (R(A)\cap R(B)))\\
	&\subseteq (A \cap R(A)) \cup (B\cap R(B)),
\end{align*}
the last of which is residually null.
Hence, $R(A) \cap R(B) \subseteq R(A\cup B)$.
Consequently, $CR(A \cup B) = CR(A) \cup CR(B)$.

\ref{MegaLemma_RelationshipWithClosure}
Notice that $X \setminus \text{cl}_X(A)$ is open and that $A \cap (X \setminus \text{cl}_X(A)) = \emptyset$ which is surely residually null.
Hence, $X \setminus \text{cl}_X(A) \subseteq R(A)$ which provides us with the fact that $CR(A) \subseteq \text{cl}_X(A)$.

\ref{MegaLemma_OpenSetClosure}
Let $U$ be open.
Since any non-empty open set is not residually null, we see that $R(U) = X \setminus \text{cl}_X(U)$.
That is, $CR(U) = \text{cl}_X(U)$.
Obviously, $U \subseteq \text{cl}_X(U) = CR(U)$ which implies that $U \subseteq ICR(U)$.
The rest follows from the fact that $CR(U) \setminus U = \text{cl}_X(U) \setminus U$, which is closed and nowhere dense.

\ref{MegaLemma_RNiffCREmpty}
If $A$ is residually null, $R(A) = X$ which implies that $CR(A) = \emptyset$.
If $CR(A) = \emptyset$, then $R(A) = X$ so we see that $A = R(A) \cap A$ which provides us with the fact that $A$ is residually null.

\ref{MegaLemma_AMinusCR}
Notice that $A \setminus CR(A) = A \cap R(A)$ which is residually null by Theorem \ref{thm:RNPointsOfARN}.

\ref{MegaLemma_AMinusICR}
Observe that $A \setminus ICR(A) \subseteq (A \setminus CR(A)) \cup (CR(A)\setminus ICR(A))$ and that $CR(A)\setminus ICR(A)$ is closed and nowhere dense.
So \ref{MegaLemma_AMinusCR} guarantees that $A \setminus ICR(A)$ is residually null

\ref{MegaLemma_Differences}
Using \ref{MegaLemma_Monotonicity} and \ref{MegaLemma_Additivity}, notice that
\begin{align*}
	CR(A)
	&= CR((A\setminus B) \cup (A\cap B))\\
	&= CR(A\setminus B) \cup CR(A\cap B)\\
	&\subseteq CR(A \setminus B) \cup CR(B).
\end{align*}
Hence,
\[
	CR(A) \setminus CR(B) \subseteq CR(A\setminus B).
\]

Both \ref{MegaLemma_MonotonicityOverIntersections} and \ref{MegaLemma_MonotonicityOverUnions} follow immediately from \ref{MegaLemma_Monotonicity}.

\ref{MegaLemma_IdempotenceOfCR}
From \ref{MegaLemma_RelationshipWithClosure} we see that $CR(CR(A)) \subseteq \text{cl}_X(CR(A)) = CR(A)$.
From \ref{MegaLemma_Differences} and \ref{MegaLemma_AMinusCR} we see that $CR(A) \setminus CR(CR(A)) \subseteq CR(A \setminus CR(A)) = \emptyset$.
Hence, $CR(A) \subseteq CR(CR(A))$ establishing that $CR(A) = CR(CR(A))$.

\ref{MegaLemma_CRMinusA}
Surely, if $CR(A)\setminus A$ is residually null, then
\[
	CR(A) \triangle A = (CR(A)\setminus A) \cup (A \setminus CR(A))
\]
is residually null using \ref{MegaLemma_AMinusCR}.
Since $CR(A)$ is closed, $A$ is an EBP-set.

Now, suppose $A$ is an EBP-set and, appealing to Lemma \ref{lem:R-set_Representations}, write $A = (U \setminus N) \cup M$ where $U$ is open and both $N$ and $M$ are residually null.
Notice that $CR(U) \setminus CR(N) \subseteq CR(U\setminus N) \subseteq CR(U)$ by \ref{MegaLemma_Differences} and \ref{MegaLemma_Monotonicity} and that $CR(M) = CR(N) = \emptyset$ by \ref{MegaLemma_RNiffCREmpty}.
Hence, $CR(U) = CR(U \setminus N)$.
It follows that, capitalizing on \ref{MegaLemma_Additivity},
\[
	CR(A) = CR( (U \setminus N) \cup M ) = CR(U).
\]
Now,
\begin{align*}
	CR(A) \setminus A
	&= CR(U) \setminus ((U \setminus N) \cup M)\\
	&= CR(U) \cap ((U \cap N^c)^c \cap M^c)\\
	&= CR(U) \cap (U^c \cup N) \cap M^c\\
	&= [(CR(U)\setminus U) \cup (CR(U) \cap N)] \cap M^c\\
	&\subseteq (CR(U) \setminus U) \cup N.
\end{align*}
By \ref{MegaLemma_OpenSetClosure}, $CR(U) \setminus U$ is meager so we see that $CR(A) \setminus A$ is residually null.

\ref{MegaLemma_CRTriangleA}
Combine \ref{MegaLemma_AMinusCR} and \ref{MegaLemma_CRMinusA}.

\ref{MegaLemma_ICRMinusA}
If $ICR(A)\setminus A$ is residually null, then
\[
	ICR(A) \triangle A = (ICR(A)\setminus A) \cup (A \setminus ICR(A))
\]
is residually null using \ref{MegaLemma_AMinusICR}.
As $ICR(A)$ is open, $A$ is an EBP-set.

Now, suppose $A$ is an EBP-set.
Then $CR(A) \setminus A$ is residually null by \ref{MegaLemma_CRMinusA}.
Since $ICR(A) \setminus A \subseteq CR(A) \setminus A$, we have that $ICR(A) \setminus A$ is residually null.

\ref{MegaLemma_ICRTriangleA}
Combine \ref{MegaLemma_AMinusICR} and \ref{MegaLemma_ICRMinusA}.

\ref{MegaLemma_RegularityOfCR}
Notice that, for $E \subseteq X$,
\begin{align*}
	E \cap (\text{cl}_X(\text{int}_X(\text{cl}_X(E))))^c
	&\subseteq \text{cl}_X(E) \cap (\text{int}_X(\text{cl}_X(E)))^c\\
	&= \text{cl}_X(E) \setminus \text{int}_X(\text{cl}_X(E))
\end{align*}
and the last set is closed and nowhere dense.
It follows that
\[
	 (\text{cl}_X(\text{int}_X(\text{cl}_X(E))))^c \subseteq R(E) \Longrightarrow CR(E) \subseteq \text{cl}_X(\text{int}_X(\text{cl}_X(E))).
\]
Now, with \ref{MegaLemma_IdempotenceOfCR} and the fact that $CR(A)$ is closed,
\[
	CR(A)
	= CR(CR(A))
	\subseteq \text{cl}_X(\text{int}_X(CR(A)))
	= \text{cl}_X(ICR(A))
	\subseteq CR(A).
\]
Conclusively, $\text{cl}_X(ICR(A)) = CR(A)$.
\end{proof}

\begin{corollary} \label{cor:ICRSeparation}
Suppose $X$ is Polish and that $A$ and $B$ are EBP-sets.
Then
\[
	ICR(A) \cap ICR(B) \neq \emptyset \Longrightarrow A \cap B \neq \emptyset.
\]
\end{corollary}
\begin{proof}
Suppose $A \cap B = \emptyset$ and notice that
\begin{align*}
	ICR(A) \cap ICR(B)
	&\subseteq [(ICR(A) \setminus A)\cup A] \cap [(ICR(B) \setminus B)\cup B]\\
	&\subseteq (ICR(A)\setminus A) \cup (ICR(B) \setminus B).
\end{align*}
Since $ICR(A) \cap ICR(B)$ is an open set and Lemma \ref{lem:MegaLemma} \ref{MegaLemma_ICRMinusA} provides us with the fact that both $ICR(A) \setminus A$ and $ICR(B)\setminus B$ are residually null, we see that $ICR(A) \cap ICR(B) = \emptyset$.
\end{proof}

\begin{lemma} \label{lem:NiceSets}
Let $X$ be a Polish space and $A \subseteq X$.
Then $A \subseteq CR(A)$ if and only if, for every open set $U \subseteq X$ with $U \cap A \neq \emptyset$, $U \cap A$ is not residually null in $X$.
\end{lemma}
\begin{proof}
Suppose that, for each open set $U \subseteq X$, $U \cap A \neq \emptyset$ implies that $U \cap A$ is not residually null.
Then, $R(A) = X \setminus \text{cl}_X(A)$.
Hence, $A \subseteq \text{cl}_X(A) \subseteq CR(A)$.

Now, assume there is an open set $U\subseteq X$ so that $U \cap A \neq \emptyset$ but $U \cap A$ is residually null.
It follows that $A \cap U \subseteq U \subseteq R(A)$.
For $x \in A \cap U$, $x \in A \cap R(A)$.
That is, $A \not\subseteq CR(A)$.
\end{proof}

A few of the results that follow generalize some of those found in \cite{Cohen}.

\begin{theorem} \label{thm:ContinuityWithNiceSets}
Suppose $X$ is a Polish space, $( Y, \rho )$ is a metric space, and $f : X \mapsto Y$ satisfies the following:
\begin{itemize}
\item $f^{-1}[B] \in EBP(X)$ for each open ball $B \subseteq Y$;
\item $f^{-1}[V] \subseteq CR(f^{-1}[V])$ for all open $V \subseteq Y$.
\end{itemize}
Then the points of continuity of $f$ form a dense $G_\delta$ subset of $X$.
\end{theorem}
\begin{proof}
Fix $n \in \omega$ and, for $x \in X$, let $V_x = B_\rho \left( f(x), 2^{-(n+1)} \right)$ and $U_x = ICR(f^{-1}[V_x])$.
Then define $\mathcal U_n = \bigcup \{ U_x : x \in X\}$ and notice that $\mathcal U_n$ is open.
We will endeavor to show $\mathcal U_n$ is dense.
Let $W \subseteq X$ be a non-empty open set and pick $x \in W$.
Observe that
\[
	x \in f^{-1}[V_x] \subseteq CR(f^{-1}[V_x]).
\]
It follows that $W \cap CR(f^{-1}[V_x]) \neq \emptyset$ and is relatively open in $CR(f^{-1}[V_x])$.
By Lemma \ref{lem:MegaLemma} \ref{MegaLemma_RegularityOfCR}, we know that $U_x$ is dense in $CR(f^{-1}[V_x])$.
Hence, $W \cap U_x \neq \emptyset$.
It follows that $\mathcal U_n$ is open and dense in $X$.

Next, we'll show that $\text{osc}_f(x) \leq 2^{-n}$ at each point $x \in \mathcal U_n$.
Towards this end, we will first see that, for any open ball $V \subseteq Y$ and $x \in {ICR}(f^{-1}[V])$, $f(x) \in \text{cl}_Y(V)$.
Otherwise, there is an open ball $W$ about $f(x)$ so that $W \cap V = \emptyset$.
Since $x \in f^{-1}[W] \subseteq CR(f^{-1}[W])$, we see that
\[
	{ICR}(f^{-1}[V]) \cap CR(f^{-1}[W]) \neq \emptyset.
\]
Alas, as $ICR(f^{-1}[W])$ is dense in $CR(f^{-1}[W])$,
\[
	{ICR}(f^{-1}[V]) \cap {ICR}(f^{-1}[W]) \neq \emptyset.
\]
Since $f^{-1}[V]$ and $f^{-1}[W]$ are EBP-sets, Lemma \ref{cor:ICRSeparation} implies that $f^{-1}[V] \cap f^{-1}[W] \neq \emptyset$ which provides $V \cap W \neq \emptyset$, a contradiction.

Now, let $x \in \mathcal U_n$ and pick $w \in X$ so that $x \in U_w$.
By the above paragraph, for $y, z \in U_w$, we have that $f(y), f(z) \in \text{cl}_Y(V_w)$.
It follows that
\[
	\rho(f(y),f(z)) \leq \rho(f(y),f(w)) + \rho(f(w),f(z)) \leq \frac{1}{2^{n+1}} + \frac{1}{2^{n+1}} = \frac{1}{2^n}.
\]
Since $x \in \mathcal U_n$ was arbitrary, we see that $\text{osc}_f(x) \leq 2^{-n}$ for every point $x \in \mathcal U_n$.

Finally, $\mathcal U = \bigcap\{ \mathcal U_n : n \in \omega\}$ is a dense $G_\delta$ for which $f$ has zero oscillation for every point of $\mathcal U$.
That is, $f$ is continuous at every point of $\mathcal U$.
\end{proof}

\begin{definition}
Let $X$ be a Polish space, $Y$ be a topological space, and $f : X \mapsto Y$ be a function.
We say that $f$ is EBP-measurable if $f^{-1}[U]$ is an EBP-set for every open set $U\subseteq Y$.
\end{definition}

\begin{theorem} \label{thm:R-measurableAndSecondCountable}
Let $X$ be a Polish space, $Y$ be a topological space with a countable
basis, and $f : X \mapsto Y$ be EBP-measurable.
Then, there exists a set $A \subseteq X$ which is co-residually null so that $f \restriction_A$ is continuous.
\end{theorem}
\begin{proof}
Let $\{ V_n : n \in \omega\}$ be a basis for $Y$ and notice that $f^{-1}[V_n]$ is an EBP-set for each $n \in \omega$ by hypothesis.
Hence, by Lemma \ref{lem:R-set_Representations}, we can write, for each $n \in \omega$,
\[
	f^{-1}[V_n] = (U_n \setminus N_n) \cup M_n
\]
where $U_n$ is open and both $N_n$ and $M_n$ are residually null.
Let
\begin{align*}
	A
	&= \left[ X\setminus \bigcup\{N_n : n\in\omega\} \right] \cap \left[  X\setminus \bigcup\{M_n : n\in\omega\} \right]\\
	&= X \setminus \left[ \bigcup\{N_n: n \in \omega\} \cup \bigcup\{M_n: n \in \omega\}\right]
\end{align*}
and notice that $A$ is co-residually null.
Now, to see that $f\restriction_A$ is continuous, notice that
\begin{align*}
	f^{-1}[V_n] \cap A
	&= [(U_n \setminus N_n) \cup M_n] \cap A\\
	&= (U_n \cap N_n^c \cap A) \cup (M_n \cap A)\\
	&= U_n \cap A
\end{align*}
and $U_n\cap A$ is relatively open in $A$, finishing the proof.
\end{proof}

\begin{proposition} \label{prop:GleasonActions}
Let $X$ be Polish, $Y$ be a separable metric space, and $G$ be a group that acts as a group of homeomorphisms on $X$ and $Y$ which is transitive on $X$.\
If $f : X \mapsto Y$ is an  map which is $G$-equivariant, then $f$ is continuous.
\end{proposition}
\begin{proof}
First, we will see that $f^{-1}[V] \subseteq CR(f^{-1}[V])$ for any open set $V \subseteq Y$.
Let $U \subseteq X$ and $V \subseteq Y$ be open so that $U \cap f^{-1}[V] \neq \emptyset$ and pick $x \in U \cap f^{-1}[V]$.
Then $\langle x , f(x) \rangle \in U \times V$.
For $y \in X$, let $g \in G$ be so that $g \cdot x = y$.
Then
\[
	\langle y, f(y) \rangle = \langle g \cdot x , f( g \cdot x) \rangle = \langle g \cdot x , g \cdot f(x) \rangle \in g\cdot U \times g \cdot V.
\]
It follows that $\text{graph}(f) \subseteq \bigcup \{ g \cdot U \times g \cdot V : g \in G \}$.
Since $\text{graph}(f)$ is separable and metrizable as it is a subset of $X \times Y$, it is Lindel\"{o}f which implies that there is $\{ g_n : n \in \omega\} \subseteq G$ so that
\[
	\text{graph}(f) \subseteq \bigcup \{ g_n \cdot U \times g_n \cdot V : n \in \omega\}.
\]
Let $X_n = \{ x \in X : \langle x , f(x) \rangle \in g_n \cdot U \times g_n \cdot V \} = g_n \cdot (U \cap f^{-1}[V])$.
Since $X = \bigcup\{ X_n : n \in \omega\}$, some $X_n$ is not residually null.
Hence, $U \cap f^{-1}[V]$ is not residually null.
Thus, by Lemma \ref{lem:NiceSets}, $f^{-1}[V] \subseteq CR(f^{-1}[V])$.

Now, by Theorem \ref{thm:ContinuityWithNiceSets}, we have that the points of continuity of $f$ is a dense $G_\delta$ subset of $X$.
To see that $f$ is continuous everywhere, let $x \in X$.
For some $y \in X$, $f$ is continuous at $y$.
As $G$ is transitive, let $g \in G$ be so that $g\cdot x = y$ and $\{ x_n : n \in \omega\} \subseteq X$ so that $x_n \to x$.
Then $g \cdot x_n \to g \cdot  x = y$.
As $f$ is $G$-equivariant, $f(g \cdot x_n) = g \cdot f(x_n)$ and notice that
\[
	g\cdot f(x_n) = f(g \cdot x_n) \to f(y) = f(g \cdot x) = g \cdot f(x)
\]
as $f$ is continuous at $y$.
It follows that $f(x_n) \to f(x)$, which is to say that $f$ is continuous at $x$, finishing the proof.
\end{proof}

\begin{corollary}
Let $G$ be a multiplicative group with a Polish topology so that, for
each $g \in G$, the map $h \mapsto gh$, $G \mapsto G$, is continuous and,
for each $h \in G$, the map $g \mapsto gh$, $G \mapsto G$, is EBP-measurable. Then $G$ is a Polish group.
\end{corollary}
\begin{proof}
Fix $h \in G$ and let $\phi : G \mapsto G$ be defined by $\phi(g) = gh$.
By assumption, $\phi$ is EBP-measurable.
Let $G$ act on itself by $\langle g,x \rangle \mapsto gx$, $G^2 \mapsto G$, and notice that $G$ acts as a group of homeomorphisms on $G$ and that this action is transitive.
Moreover, $\phi$ is $G$-equivariant as, for any $g ,x \in G$, $\phi(gx) = gxh = g \phi(x)$.
Hence, $\phi$ is continuous by Proposition \ref{prop:GleasonActions}.
It follows that multiplication is separately continuous so, applying \cite{montgomery}, we are done.
\end{proof}

\section{The Alexandrov-Suslin Operation}

We will now discover that EBP-sets are closed under the Alexandrov-Suslin operation, hereinafter referred to as the $\mathcal A$-operation, which was introduced independently by P. Alexandrov in \cite{PAlex} and M. Suslin in \cite{Suslin}.
One of the properties of the $\mathcal A$-operation is that all $\mathbf \Sigma^1_1$ subsets of a Polish space can be obtained by the $\mathcal A$-operation on a countable family of closed sets.
It's also true that countable families of Baire property sets are preserved under the $\mathcal A$-operation.
Let's recall the definition.

\begin{definition}
Let $X$ be a set.
By an $\mathcal A$-system, we mean a function $S : \omega^{<\omega} \mapsto \wp(X)$.
For an $\mathcal A$-system $S$, we define the $\mathcal A$-operation on $S$ to be
\[
	\mathcal A(S) = \bigcup_{w \in \omega^\omega} \bigcap_{n \in \omega} S(w \restriction_n).
\]
We say that an $\mathcal A$-system $S$ is regular if $S(w\restriction_{n+1}) \subseteq S(w \restriction_n)$ for all $w\in \omega^{<\omega}$ and $n \in \omega$.
\end{definition}
\begin{lemma} \label{lem:RegularSystems}
With respect to the $\mathcal A$-operation, all $\mathcal A$-systems can be assumed to be regular without loss of generality.
That is, for any $\mathcal A$-system $S$, there is a regular $\mathcal A$-system $S'$ so that $\mathcal A(S) = \mathcal A(S')$.
\end{lemma}
\begin{proof}
For each $\sigma \in \omega^{<\omega}$ of length $n$, define
\[
	S'( \sigma ) = \bigcap \{ S( \sigma\restriction_m ) : m \leq n \}
\]
and notice that $S'$ is the desired system.
\end{proof}

\begin{lemma} \label{lem:CoveringLemma}
Let $X$ be a Polish space and $A \subseteq X$.
There exists an EBP-set $B$ so that, for any EBP-set $E$, if $A \subseteq E$, $B \setminus E$ is residually null.
If desired, this $B$ can be chosen to be the union of a closed set with a residually null set.
\end{lemma}
\begin{proof}
Let $F = CR(A)$, $R = (A \setminus CR(A))$, $B = F \cup R$, and notice that $A \subseteq B$.
Since $A \setminus CR(A)$ is residually null by Lemma \ref{lem:MegaLemma} \ref{MegaLemma_AMinusCR}, Lemma \ref{lem:R-set_Representations} guarantees that $B$ is an EBP-set.

Now, suppose $E$ is any EBP-set so that $A \subseteq E$.
By Lemma \ref{lem:R-set_Representations}, we can write $E$ as $(U \setminus N) \cup M$ where $U$ is open and both $N$ and $M$ are residually null.
Now,
\begin{align*}
	(F \cup R) \setminus ((U \setminus N)\cup M)
	&= (F\cup R) \cap ((U\cap N^c)^c \cap M^c)\\
	&= (F\cup R) \cap ((U^c \cup N) \cap M^c)\\
	&\subseteq [F \cap ((U^c \cup N) \cap M^c)] \cup R\\
	&\subseteq (F \cap U^c \cap M^c) \cup N \cup R\\
	&= [F \setminus (U \cup M)] \cup N \cup R.
\end{align*}
So we just need to see that $F \setminus (U \cup M)$ is residually null.
Observe that, appealing to Lemma \ref{lem:MegaLemma},
\[
	A \subseteq U \cup M \Longrightarrow CR(A) \subseteq CR(U\cup M) = CR(U) \Longrightarrow F \subseteq \text{cl}_X(U).
\]
Thus,
\[
	F \setminus (U \cup M) \subseteq \text{cl}_X(U) \setminus (U \cup M) \subseteq \text{cl}_X(U) \setminus U,
\]
finishing the proof.
\end{proof}

\begin{theorem}
The class of EBP-sets is closed under the $\mathcal A$-operation.
That is, suppose $S$ is an $\mathcal A$-system so that, for every $\sigma \in \omega^{<\omega}$, $S(\sigma)$ is an EBP-set.
Then $\mathcal A(S)$ is an EBP-set.
\end{theorem}

\begin{proof}
Apply Lemma \ref{lem:CoveringLemma} and \cite[Theorem 29.13]{kechris}.
\end{proof}

\section{Applications to Polish Groups}

\begin{proposition} \label{prop:coResiduallyNullGroupHom_Continuous}
Let $G$ be a multiplicative Polish group, $H$ be a multiplicative topological group, and $\phi: G \mapsto H$ be a group homomorphism.
If there exists a set $A \subseteq X$ which is co-residually null so that $\phi \restriction_A$ is continuous, then $\phi$ is continuous.
\end{proposition}
\begin{proof}
Let $\{ g_n : n \in \omega \} \subseteq G$ be so that $g_n \to g \in G$.
Since $A$ is co-residually null, $g^{-1}A \cap \bigcap\{ g_n^{-1} A : n \in \omega\}$ is co-residually null and, in particular, non-empty.
So, pick $h \in g^{-1}A \cap \bigcap\{ g_n^{-1} A : n \in \omega\}$.
It follows that $gh \in A$ and, for each $n \in \omega$, $g_n h \in A$.
Since $\phi\restriction_A$ is continuous and $g_n h \to gh$,
\[
	\phi(g_n)\phi(h) = \phi(g_n h) \to \phi(gh) = \phi(g) \phi(h).
\]
It follows that $\phi(g_n) \to \phi(g)$, which is to say that $\phi$ is continuous.
\end{proof}

\begin{corollary}
Let $G$ be a multiplicative Polish group, $H$ be a multiplicative topological group with a countable basis, and $\phi : G \mapsto H$ be a group homomorphism which is EBP-measurable.
Then $\phi$ is continuous.
\end{corollary}
\begin{proof}
Use Theorem \ref{thm:R-measurableAndSecondCountable} and Proposition \ref{prop:coResiduallyNullGroupHom_Continuous}.
\end{proof}

Now we provide a generalization to the famous theorem of B. J. Pettis in \cite{Pettis}.

\begin{theorem} \label{thm:R-setPettis}
Suppose $G$ is a multiplicative Polish group and suppose $A$ is an EBP-set of $G$ which is not residually null.
Then both $A^{-1}A$ and $AA^{-1}$ contain a neighborhood of the identity.
\end{theorem}
\begin{proof}
Let $U$ be open so that $A \triangle U$ is residually null and notice that $U \neq \emptyset$ since $A$ is assumed to not be residually null.
Pick $g \in U$ and find a neighborhood of the identity $V$ so that $gVV^{-1} \subseteq U$.
Let $h, p \in V$ and notice that $gph^{-1} \in U$ which implies that $gp \in Uh$.
Also, since $V$ is a neighborhood of identity, $gp \in gVV^{-1} \subseteq U$.
As $p \in V$ was arbitrary, we see that $gV \subseteq U \cap Uh$.

For $h \in V$,
\[
	(U \cap Uh) \triangle (A \cap Ah) \subseteq (U \triangle A) \cup (Uh \triangle Ah) = (U\triangle A) \cup ((U\triangle A)h).
\]
Since multiplication on the right is a homeomorphism and the notion of being residually null is a topological invariant by Proposition \ref{prop:TopologicalInvarianceOfRN}, we see that
\[
	(U \cap Uh) \triangle (A \cap Ah)
\]
is residually null.
Since $U \cap Uh$ is a non-empty open set, $U \cap Uh$ is not residually null by Lemma \ref{lem:open_positive_dense} so $A \cap Ah \neq \emptyset$.

So, let $h \in V$ and, since $A\cap Ah^{-1} \neq \emptyset$, let $g \in A \cap Ah^{-1}$.
It follows that $g^{-1} \in A^{-1}$ and $gh \in A$.
Thus, $h = g^{-1}gh \in A^{-1}A$.
That is, $V \subseteq A^{-1} A$ so we see that $A^{-1}A$ contains a neighborhood of identity.
A similar argument establishes the same result for $AA^{-1}$.
\end{proof}

\begin{corollary}
Let $G$ be a multiplicative Polish group and suppose $H$ is a subgroup which is an EBP-set but not residually null.
Then $H$ is open.
\end{corollary}
\begin{proof}
Since $H$ is an EBP-set which is not residually null, by Theorem \ref{thm:R-setPettis}, we see that $HH^{-1}$ contains a neighborhood $U$ of the identity.
So $U \subseteq HH^{-1} \subseteq H$ and we see that $H = \bigcup \{ hU : h \in H \}$ establishing that $H$ is open.
\end{proof}

\begin{theorem} \label{thm:R-setContinuity}
Let $G$ be a multiplicative Polish group, $H$ be a multiplicative separable group, and $\phi : G \mapsto H$ be a homomorphism which is EBP-measurable.
Then $\phi$ is continuous.
\end{theorem}
\begin{proof}
It suffices to show that $\phi$ is continuous at $1_G$ so let $U$ be a neighborhood of $1_H$.
Then pick a neighborhood $V$ of $1_H$ so that $V^{-1}V \subseteq U$.
Let $\{ h_n : n \in \omega\} \subseteq H$ be dense and notice that, for any $h \in H$, $hV^{-1}$ is a neighborhood of $h$.
Pick $h_n$ so that $h_n \in hV^{-1}$ and $p \in V^{-1}$ so that $h_n = hp$.
Then $h = h_n p^{-1} \in h_n V$.
So we see that $H = \bigcup \{ h_nV : n \in \omega\}$.

Now, $G = \bigcup\{ \phi^{-1}[h_n V] : n \in \omega\}$.
It follows that there must be some $n \in \omega$ so that $\phi^{-1}[h_nV]$ is not residually null.
By assumption, $\phi^{-1}[h_nV]$ is an EBP-set so Theorem \ref{thm:R-setPettis} provides us with the fact that
\[
	(\phi^{-1}[h_nV])^{-1}\phi^{-1}[h_nV]
\]
contains a neighborhood of $1_G$.
Observe that
\begin{align*}
	(\phi^{-1}[h_nV])^{-1}\phi^{-1}[h_nV]
	&= \phi^{-1}[(h_nV)^{-1}]\phi^{-1}[h_nV]\\
	&\subseteq \phi^{-1}[(h_nV)^{-1} h_nV]\\
	&= \phi^{-1}[V^{-1}h_n^{-1}h_n V]\\
	&= \phi^{-1}[V^{-1}V]\\
	&\subseteq \phi^{-1}[U],
\end{align*}
finishing the proof.
\end{proof}

\begin{corollary}
Let $G$ and $H$ be multiplicative Polish groups and $\phi : G \mapsto H$ be a homomorphism which is EBP-measurable.
Then $\phi$ is continuous.
If moreover, $\phi[G]$ is not residually null, $\phi$ is open.
\end{corollary}

\begin{proof}
Immediately, by Theorem \ref{thm:R-setContinuity}, $\phi$ is continuous.
So suppose $\phi[G]$ is not residually null and fix an open set $W \subseteq G$ with $1_G \subseteq W$.
Let $V \subseteq G$ be open so that $1_G \in V$ and $V^{-1}V \subseteq W$.
Let $\{ g_n : n \in \omega \}$ be a dense subset of $G$ and notice that $G = \bigcup \{ g_n V : n \in \omega \}$.
It follows that
\[
	\phi[G] = \phi\left[  \bigcup \{ g_n V : n \in \omega \} \right] = \bigcup\{ \phi(g_n) \phi[V] : n \in \omega \}
\]
and, since $\phi[G]$ is not residually null, there is some $n \in \omega$ so that $\phi(g_n) \phi[V]$ is not residually null.
Hence, $\phi[V]$ is not residually null and, as the continuous image of an open set, $\phi[V]$ is also an analytic subset of $H$ so $\phi[V]$ is an EBP-set.
Thus, applying Theorem \ref{thm:R-setPettis}, we see that there exists some open $W' \subseteq H$ so that
\[
	1_H \in W' \subseteq (\phi[V])^{-1}\phi[V] = \phi[V^{-1}] \phi[V] \subseteq \phi[V^{-1} V] \subseteq \phi[W].
\]

Now, for any non-empty open set $U \subseteq G$, let $y \in \phi[U]$ and pick $x \in U$ so that $\phi(x) = y$.
Find an open set $W_x$ so that $1_G \in W_x$ and $x W_x \subseteq U$.
By the above paragraph, we can find an open set $V_y \subseteq H$ so that $1_H \in V_y \subseteq \phi[W_x]$.
Then
\[
	y \in y V_y \subseteq \phi(x) \phi[W_x] = \phi[ x W_x ] \subseteq \phi[U]
\]
which guarantees that
\[
	\phi[U] = \bigcup \{ yV_y : y \in \phi[U] \},
\]
establishing that $\phi[U]$ is open.
\end{proof}

\section{Martin's Axiom and Residually Null Sets}
\label{section:MA}

Recall that a topological space $X$ has countable cellularity provided that, for any family $\mathscr U$ of pair-wise disjoint open sets, $\#\mathscr U \leq \aleph_0$.
For a set $A$ and a cardinal $\kappa$, we will say that $A$ is $\kappa$-sized provided $\# A \leq \kappa$.

\begin{lemma} \label{lem:CountableCellularity}
Suppose $X$ is any topological space and $D \subseteq X$ be a dense subspace which is of countable cellularity with respect to its inherited topology.
Then $X$ is of countable cellularity.
\end{lemma}
\begin{proof}
Let $\mathscr U$ be a family of pair-wise disjoint open subsets of $X$ and define $\mathscr U_D = \{ U \cap D : U \in \mathscr U\}$.
Notice that $\mathscr U_D$ is a family of pair-wise disjoint open subsets of $D$ and that, since $D$ is dense in $X$, $U \cap D \neq \emptyset$ for each $U \in \mathscr U$.
As $D$ is assumed to be of countable cellularity, we see that $X$ is also of countable cellularity.
\end{proof}

The following is the topological equivalent to the classical Martin's Axiom.

\begin{definition}
For a cardinal $\kappa$, let $\text{MA}(\kappa)$ be the statement:
For any compact Hausdorff space $X$ with countable cellularity and for any $\kappa$-sized collection $\mathscr U$ of open dense subsets of $X$, $\bigcap \mathscr U \neq \emptyset$.
\end{definition}

In particular, $\text{MA}(\omega)$ is true without the condition of countable cellularity and is equivalent to the Baire Category Theorem.
Also, $\text{MA}(\mathfrak c)$ is false by looking at $[0,1]$ and, for each $x \in [0,1]$, defining $U_x = [0,1]\setminus\{x\}$.
Then the family $\mathscr U := \{ U_x : x \in [0,1] \}$ is a $\mathfrak c$-sized collection of open dense subsets but $\bigcap \mathscr U = \emptyset$.
From this, we see that Martin's Axiom is only non-trivial in models of $\neg\text{CH}$.
In general, we will refer to Martin's Axiom as the assertion that $\text{MA}(\kappa)$ holds for all cardinals $\kappa < \mathfrak c$.

Recall that a topological space $X$ is \v{C}ech-complete if $X$ admits a compactification $K$ in which $X$ is a $G_\delta$ subset of $K$.
In fact, if a topological space $X$ is \v{C}ech-complete, $X$ is actually a $G_\delta$ subset of any of its compactifications (\cite[Theorem 3.9.1]{Engelking}).

\begin{proposition} \label{prop:MA_equivalents}
For any cardinal $\kappa$, $\text{MA}(\kappa)$ is equivalent to the statement:
For any \v{C}ech-complete space $X$ of countable cellularity and any $\kappa$-sized family $\mathscr U$ of open dense sets, $\bigcap \mathscr U \neq \emptyset$;
\end{proposition}

\begin{proof}
Since any compact Hausdorff space $X$ is trivially \v{C}ech-complete, we need only see that $\text{MA}(\kappa)$ implies the proposed equivalent.

Suppose $X$ is a \v{C}ech-complete space of countable cellularity and let $\mathscr U$ be a $\kappa$-sized family of open dense sets where $\kappa < \mathfrak c$.
Choose a compactification $K$ of $X$ and notice that, by Lemma \ref{lem:CountableCellularity}, $K$ is also of countable cellularity.
Then, for each $U \in \mathscr U$, let $V_U \subseteq K$ be open so that $X \cap V_U = U$ and define $\mathscr V = \{ V_U : U \in \mathscr U \}$.
As each $U \in \mathscr U$ is dense in $X$ and $X$ is dense in $K$, we see that each $V_U$ is an open dense subset of $K$.

Now, as $X$ is a dense $G_\delta$ in $K$, let $\mathscr W$ be a countable family of open dense subsets of $K$ so that $\bigcap \mathscr W = X$.
As long as $\kappa$ is an infinite cardinal, $\mathscr W \cup \mathscr V$ is a $\kappa$-sized collection of open sets so we can apply $\text{MA}(\kappa)$ to see that
\[
	\emptyset
	\neq \bigcap \mathscr W \cap \bigcap \mathscr V
	= X \cap \bigcap \mathscr V
	= \bigcap \mathscr U,
\]
finishing the proof.
\end{proof}

Let $X$ be Polish and recall that $X$ can be homeomorphically embedded into $[0,1]^\omega$.
As $[0,1]^\omega$ is also Polish, Mazurkiewicz' Theorem guarantees that $X$ is \v{C}ech-complete and of countable singularity as $X$ is separable.

\begin{corollary} \label{cor:kappa_meager}
Let $X$ be a Polish space and $\mathscr E$ be a $\kappa$-sized family of meager sets where $\kappa < \mathfrak c$.
Assuming Martin's Axiom, $\bigcup \mathscr E$ is meager.
\end{corollary}

\begin{theorem}[{\cite[Exercise III.3.30]{Kunen2011}}] \label{thm:kappa_null}
Let $X$ be a Polish space, $\mu$ be a Borel probability measure on $X$, and $\mathscr E$ be a $\kappa$-sized family of $\mu$-null sets where $\kappa < \mathfrak c$.
Assuming Martin's Axiom, $\bigcup \mathscr E$ is $\mu$-null.
\end{theorem}

\begin{theorem} \label{thm:MAForRN}
Let $X$ be a Polish space and $\mathscr E$ be a $\kappa$-sized family of residually null sets where $\kappa < \mathfrak c$.
Assuming Martin's Axiom, $\bigcup \mathscr E$ is residually null.
\end{theorem}
\begin{proof}
Let $\mathscr E = \{ E_\xi : \xi < \kappa \}$ be a collection of residually null sets and notice that, by Corollary \ref{cor:kappa_meager}, $\bigcap \{ N(E_\xi) : \xi < \kappa \}$ is co-meager in $\mathcal M(X)$.
Let $\mu \in \bigcap \{ N(E_\xi) : \xi < \kappa \}$ and notice that, by Theorem \ref{thm:kappa_null}, $\mu(\bigcup \mathscr E) = 0$.
Hence, $\bigcap \{ N(E_\xi) : \xi < \kappa \} \subseteq N(\bigcup \mathscr E)$ which affirms that $\bigcup \mathscr E$ is residually null.
\end{proof}

\begin{corollary}
Let $X$ be a Polish space and $\mathscr E$ be a $\kappa$-sized family of EBP-sets where $\kappa < \mathfrak c$.
Assuming Martin's Axiom, $\bigcup \mathscr E$ is an EBP-set.
\end{corollary}
\begin{proof}
Let $\mathscr E = \{ E_\xi : \xi < \kappa \}$ be a collection of EBP-sets and pick $U_\xi$ open so that $E_\xi \triangle U_\xi$ is residually null for each $\xi < \kappa$.
Then, observe that
\[
	\left[ \bigcup\{E_\xi : \xi < \kappa \} \right] \triangle \left[ \bigcup\{U_\xi : \xi < \kappa \} \right]
	\subseteq \bigcup \{ E_\xi \triangle U_\xi : \xi < \kappa \}
\]
which, by appealing to Theorem \ref{thm:MAForRN}, completes the proof.
\end{proof}

Recall that the density of a topological space $X$ is the least cardinal $\kappa$ so that there exists a $\kappa$-sized dense subset of $X$.
We will let $\text{den}(X)$ denote the density of $X$.

\begin{corollary}
Let $G$ be a Polish group, $H$ be a topological group with $\text{den}(H) < \mathfrak c$, and $\phi : G \mapsto H$ be a homomorphism which is EBP-measurable.
Assuming Martin's Axiom, $\phi$ is continuous.
\end{corollary}
\begin{proof}
In the proof of Theorem \ref{thm:R-setContinuity}, replace the dense set $\{ h_n : n \in \omega\}$ with a dense set $\{ h_\xi : \xi < \kappa \}$.
Then use Theorem \ref{thm:MAForRN} to see that there has to be some $\xi < \kappa$ so that $\phi^{-1}[h_\xi V]$ is not residually null.
\end{proof}

\providecommand{\bysame}{\leavevmode\hbox to3em{\hrulefill}\thinspace}
\providecommand{\MR}{\relax\ifhmode\unskip\space\fi MR }
\providecommand{\MRhref}[2]{%
  \href{http://www.ams.org/mathscinet-getitem?mr=#1}{#2}
}
\providecommand{\href}[2]{#2}

\end{document}